\documentclass[10pt]{amsart}
\usepackage{latexsym, amsmath,amssymb,bm}

\usepackage{xcolor}

\renewcommand{\epsilon}{\varepsilon}  
\newcommand{\newsection}[1]
{\subsection{#1}\setcounter{theorem}{0}  \setcounter{equation}{0}
\par\noindent}

\newtheorem{theorem}{Theorem}

\newtheorem{lemma}[theorem]{Lemma}
\newtheorem{corr}[theorem]{Corollary}

\newtheorem{proposition}[theorem]{Proposition}
\newtheorem{deff}[theorem]{Definition}

\newcommand{\bth}{\begin{theorem}}
\newcommand{\ble}{\begin{lemma}}
\newcommand{\bcor}{\begin{corr}}

\newcommand{\bdeff}{\begin{deff}}

\newcommand{\bprop}{\begin{proposition}}
\newcommand{\ele}{\end{lemma}}
\newcommand{\ecor}{\end{corr}}
\newcommand{\edeff}{\end{deff}}

\newcommand{\eprop}{\end{proposition}}

\newcommand{\Rn}{{\mathbb R}^n}

\newcommand{\la}{\lambda}

\newcommand{\e}{\varepsilon}

\renewcommand{\Pi}{\varPi}

\renewcommand{\epsilon}{\varepsilon}

\newcommand{\dist}{{{\rm dist}}}

\newcommand{\R}{{\mathbb R}}

\newcommand{\one}{{\bf 1}}

\begin{document}

\title
{Quasimode concentration on compact space forms}

\thanks{Both authors were supported in part by the Simons Foundation,
and the  second author was supported in part by the NSF (DMS-1665373). }


\author{Xiaoqi Huang}
\address[X.H.]{Department of Mathematics, Louisiana State University, Baton Rouge, LA 70803}
\email{xhuang49@lsu.edu
}
\author{Christopher D. Sogge}
\address[C.D.S.]{Department of Mathematics,  Johns Hopkins University,
Baltimore, MD 21218}
\email{sogge@jhu.edu}

\begin{abstract}
We show that the upper bounds for the $L^2$-norms of  $L^1$-normalized quasimodes that we obtained
in \cite{HSqm} are always sharp on any compact space
form.  This allows us to characterize compact
manifolds of constant sectional curvature using
the decay rates of lower bounds of $L^1$-norms of
$L^2$-normalized log-quasimodes fully resolving
a problem initiated by the second author
and Zelditch~\cite{SoggeZelNodal}.  We are also able 
to characterize such manifolds by the concentration of quasimodes
near periodic geodesics as measured by $L^2$-norms
over thin geodesic tubes.
\end{abstract}

\maketitle

\centerline{ \bf In memoriam: {\em Steve Zelditch (1953-2022)}}

\newsection{Introduction and main results}

There are many ways of measuring the concentration of eigenfunctions and quasimodes
on compact manifolds $(M,g)$ some of which may be sensitive to the underlying geometry, such as
curvature assumptions.  See, e.g., \cite{SoggeCon}.  A standard way is through the growth rate
of the $L^q(M)$ norms of $L^p$-normalized modes.  Since the manifold is compact this is only
of interest when $q>p$.  Another method, which has been
studied by several authors over the last
decade or so, starting with the second author
and Zelditch~\cite{SoggeZelditchL4}, is through
$L^2$-norms over small sets, especially geodesic
tubes.

Typically one works with $L^2$-normalized modes and so,
for the first way of measurement, one often takes $p=2$.  There now are sharp estimates in this
case for log-quasimodes through the works of B\'erard~\cite{Berard}, Hassell and Tacy~\cite{HassellTacy}
and the authors~\cite{HSqm}.   

To describe these results, if $\delta(\la)\in (0,1]$ let
\begin{equation}\label{1.1}
V_{[\la,\la+\delta(\la)]}=\{\Phi_\la: \, \text{Spec }\Phi_\la \subset [\la,\la+\delta(\la)]\},
\end{equation}
be the space of $\delta(\la)$-quasimodes,
where $\text{Spec }$ refers to the spectrum of the first order operator
$$P=\sqrt{-\Delta_g},$$
with $\Delta_g$ being the Laplace-Beltrami operator associated to the metric $g$.  
Abusing notation a bit, we shall say that $\Phi_\la$ is a log-quasimode if $\delta(\la)=(\log\la)^{-1}$.
 Also, we typically assume
that $\delta(\la)\searrow 0$ and also that $\la\to \delta(\la)\cdot \la$ is non-decreasing.

Also recall that if
\begin{equation}\label{1.2}
q_c=\tfrac{2(n+1)}{n-1}
\end{equation}
and
\begin{equation}\label{1.3}
\mu(q)=
\begin{cases}
n(\tfrac12-\tfrac1q)-\tfrac12, \, \, \, q>q_c
\\
\tfrac{n-1}2(\tfrac12-\tfrac1q), \, \, \, 2<q\le q_c,
\end{cases}
\end{equation}
then the second author showed in \cite{sogge88} that one always has the universal estimates for unit-band modes
saying that for $\la\ge2$
\begin{equation}\label{1.4}
\|\Phi_\la\|_{L^q(M)}\lesssim \la^{\mu(q)}\|\Phi_\la\|_{L^2(M)}, \quad
\text{if } \, \, \Phi_\la \in V_{[\la,\la+1]}.
\end{equation}
These estimates are always optimal (see \cite{SFIO2}) and they also are saturated by eigenfunctions
on the round sphere, $S^n$, (see \cite{sogge86}).  On the other hand, under certain curvature assumptions, one has
improvements for $\delta(\la)$-quasimodes if $\delta(\la)\searrow 0$.

For relatively large exponents, if $\mu(q)$ is as in \eqref{1.3} one has for $\la\ge 2$
\begin{multline}\label{1.5}
\|\Phi_\la\|_q\le C_q\la^{\mu(q)}\sqrt{\delta(\la)}\|\Phi_\la\|_2, \, \,
\text{for }\, q>q_c \, \, \text{if } \, \Phi_\la\in V_{[\la,\la+\delta(\la)]}, \, \, \delta(\la)\in [(\log\la)^{-1},1],
\\
\text{and if all the sectional curvatures of } \, M \, \, \text{are nonpositive}.
\end{multline}
This result is due to B\'erard~\cite{Berard} for $q=\infty$ and to 
Hassell and Tacy~\cite{HassellTacy} for all other exponents $q>q_c$.  It is optimal in the sense
that $o(\la^{\mu(q)}\sqrt{\delta(\la)})$ bounds are not possible by an easy argument if 
$\delta(\la)\searrow 0$.  On the other hand, one expects the bounds in \eqref{1.5} to hold
for $\delta(\la)<(\log\la)^{-1}$, and this is the case, for instance, for tori (see e.g. \cite{BSSY}).
The estimates in \eqref{1.5} are saturated by quasimodes concentrating near points, which accounts
for the fact that they involve relatively large exponents.  Also, these bounds do not distinguish
between manifolds of strictly negative sectional curvatures from flat ones.

Recent work of the authors \cite{HSqm} treats the complementary range $q\in (2,q_c]$ of relatively
small exponents.  In this case the estimates that were obtained are sensitive to the sign of the
curvature and are saturated by quasimodes concentrating near periodic geodesics on
compact space forms (manifolds of constant sectional curvature).  Specifically, in \cite{HSqm} it was shown that
\begin{multline}\label{1.6}
\|\Phi_\la\|_q\le C(\la \delta(\la))^{\mu(q)}
\|\Phi_\la\|_2, \, \,
\text{for }\, q\in (2,q_c],
\\
\text{if all the sectional curvatures of } \, M \, \, \text{are nonpositive},
\end{multline}
and, moreover,
\begin{multline}\label{1.7}
\|\Phi_\la\|_q\le C_q \la^{\mu(q)} \sqrt{\delta(\la)} 
\|\Phi_\la\|_2, \, \,
\text{for }\, q\in (2,q_c],
\\
\text{if all the sectional curvatures of } \, M \, \, \text{are negative},
\end{multline}
assuming, as in \eqref{1.5}
\begin{equation}\label{1.8}
\Phi_\la \in V_{[\la,\la+\delta(\la)]}
\quad \text{with } \, \, 
\delta(\la)\in [(\log\la)^{-1},1].
\end{equation}
Note that the bounds in \eqref{1.7} are stronger than those in \eqref{1.6} since $\mu(q)<1/2$ for $q\in (2,q_c]$.

The authors in \cite{HSqm} were also able to show that these estimates are always sharp on compact
space forms, which provided a classification of compact manifolds of constant sectional curvature $K$
in terms of the sign of the curvature and the growth rate of log-quasimodes.
Specifically, it was shown that for $q\in (2,q_c]$
\begin{multline}\label{1.9}
\sup \bigl\{ \|\Phi_\la\|_{L^q(M)}: \, \Phi_\la \in V_{[\la,\la+(\log)^{-1}]}, \, 
\|\Phi_\la\|_{L^2(M)}=1 \bigr\} 
\\
=\begin{cases}
\Theta(\la^{\mu(q)}(\log\la)^{-1/2}) \, \, \iff \, \, K<0
\\
\Theta(\la^{\mu(q)}(\log\la)^{-\mu(q)}) \, \, \iff \, \, K=0
\\
\Theta(\la^{\mu(q)}) \, \, \iff \, \, K>0,
\end{cases}
\end{multline}
if $(M,g)$ is a compact space form all of whose sectional curvatures equal $K$.
Here we are taking the left side of \eqref{1.9} to be zero if $V_{[\la,\la+(\log)^{-1}]}=\emptyset$.  
It was also shown that analogous results are valid if $(\log\la)^{-1}$ is replaced with
$\delta(\la)\in (0,1]$ satisfying $\delta(\la)\searrow 0$.  Recall that if $f,g$ are nonnegative then
we say that $f(\la)=\Theta(g(\la))$ if $f(\la)=O(g(\la))$ and also $f(\la)=\Omega(g(\la))$, with the latter being
the negation of $f(\la)=o(g(\la))$.

The $O$-bounds which are implied by \eqref{1.9} follow from \eqref{1.4} ($K>0$), 
\eqref{1.6} ($K=0$) and \eqref{1.7} ($K<0$). The $\Omega$-bounds for $K<0$ are elementary, and for $K>0$ they
just follow from the fact that there are gaps of length one in the spectrum of $\sqrt{-\Delta_{S^n}}$.  The flat case,
$K=0$, was more difficult to handle.  It was done in a somewhat circuitous manner by a Knapp-type construction, which
was reminiscent of arguments in Brooks~\cite{BrooksQM} and Sogge and Zeldtich~\cite{SoggeZelditchL4}.  We shall 
give more direct arguments that can be used to prove the $\Omega$-bounds in \eqref{1.9} when $K>0$ or $K=0$ that we shall
use to obtain sharp lower bounds of $L^1$-norms of $L^2$-normalized spectrally localized log-quasimodes on flat or positively 
curved compact space forms.

Specifically, these new Knapp-type constructions will be the main step in proving the following results which provide another
way of characterizing compact space forms.

\begin{theorem}\label{thm}
Let $(M,g)$ be an $n$-dimensional compact Riemannian manifold and assume that $\delta(\la)\searrow 0$ satisfies
$\delta(\la)\in [\log\la)^{-1},1 ]$ and $\la\to \la\cdot \delta(\la)$ is nondecreasing.
Then for $\la\ge2$
\begin{multline}\label{1.10}
\sup \bigl\{ \|\Phi_\la\|_{L^2(M)}: \, \Phi_\la \in V_{[\la,\la+\delta(\la)]}, \, 
\|\Phi_\la\|_{L^1(M)}=1 \bigr\} 
\\
=
\begin{cases}
O(\la^{\frac{n-1}4} (\delta(\la))^{N}) \, \forall \, N, \, \,
\text{if all the sectional curvatures of }M \, \text{are negative}
\\
O(\la^{\frac{n-1}4}(\delta(\la))^{\frac{n-1}4}) \, \,
\text{if all the sectional curvatures of }M \, \text{are nonpositive}
\\
O(\la^{\frac{n-1}2}) \, \, \text{for any } \, M.
\end{cases}
\end{multline}
Moreover, if $(M,g)$ is a compact space form with sectional curvatures equal to $K$ then 
\begin{multline}\label{1.11}
\sup \bigl\{ \|\Phi_\la\|_{L^2(M)}: \, \Phi_\la \in V_{[\la,\la+\delta(\la)]}, \, 
\|\Phi_\la\|_{L^1(M)}=1 \bigr\} 
\\
=
\begin{cases}
O(\la^{\frac{n-1}4}(\delta(\la))^{N}) \, \forall N \, \iff \, K<0
\\
\Theta(\la^{\frac{n-1}4}(\delta(\la))^{\frac{n-1}4}) \, \iff \, K=0
\\
\Theta(\la^{\frac{n-1}4}) \, \iff \, K>0.
\end{cases}
\end{multline}
\end{theorem}

Earlier we mentioned the problem of detecting geometric properties of compact space forms using the growth
rate of $L^q$-norms of $L^p$ normalized quasimodes.  In our earlier work \cite{HSqm} we showed that this was
possible for $p=2$ provided that, as in \eqref{1.9}, $q$ belongs to the intervals $(2,\tfrac{2(n+1)}{n-1}]$ which shrink to the 
empty set as $n\to \infty$.  On the other hand, we note that \eqref{1.11} allows us to detect the sign of the  curvature of a compact
space form using {\em the same pair} $(p,q)=(1,2)$ in {\em all dimensions}.  As we shall see, this pair of Lebesgue exponents captures very
different types of concentration near periodic geodesics for the three geometries.

The upper bounds \eqref{1.10} follow from H\"older's inequality along with \eqref{1.4}, \eqref{1.6} and \eqref{1.7},
as was shown in \cite{HSqm} and \cite{SoggeZelNodal}.  For the sake of completeness, though, let us present
the simple argument here.

Suppose that $0 \ne \Phi_\la \in V_{[\la, \la+\delta(\la) ]}$.  Then by
H\"older's inequality
$$\|\Phi_\la\|_2 \le \|\Phi_\la\|_1^{\theta_q} \,
\|\Phi_\la\|_q^{1-\theta_q}, \quad \theta_q=\tfrac{q-2}{2(q-1)}.
$$
Thus, by \eqref{1.7}, if all of the sectional curvatures of $(M,g)$ are negative
\begin{align*}
\|\Phi_\la\|_2 &\lesssim \|\Phi_\la\|_1^{\theta_q} \, \bigl( \, \la^{\frac{n-1}2(\frac12-\frac1q)}\, (\delta (\la))^{\frac12}
\| \Phi_\la\|_2\, \bigr)^{1-\theta_q}
\\
&=\bigl(\la^{\frac{n-1}4} \|\Phi_\la \|_1\bigr)^{\theta_q} \, (\delta(\la))^{\frac{(1-\theta_q)}2} \|\Phi_\la\|_2^{1-\theta_q},
\end{align*}
since
$$(1-\theta_q)\cdot \tfrac{n-1}2(\tfrac12 -\tfrac1q) =\tfrac{q}{2(q-1)}\cdot \tfrac{n-1}4(\tfrac{q-2}q)
=\tfrac{n-1}4\cdot \tfrac{q-2}{2(q-1)}=\tfrac{n-1}4 \theta_q.
$$

Thus, if all the sectional curvatures of $(M,g)$ are negative, we have for $q\in (2,q_c]$
$$\|\Phi_\la\|_2 \lesssim 
\la^{\frac{n-1}4} \, (\delta(\la))^{\frac{(1-\theta_q)}{2\theta_q}} \, \|\Phi_\la\|_1,$$
leading to the first part of \eqref{1.11} since $\theta_q \searrow 0$
as $q \searrow 2$.

We obtain the second part of \eqref{1.10} from this argument if we use \eqref{1.6}.  The last inequality
in \eqref{1.10}, which is due to Sogge and Zelditch~\cite{SoggeZelNodal}, similarly follows from
the universal bounds \eqref{1.4}.

Since we have established \eqref{1.10}, in order to 
complete the proof of Theorem~\ref{thm}, it suffices to 
prove the $\Omega$-lower bounds for $K=0$ and
$K>0$, which are implicit in \eqref{1.11}.  We cannot
use the simpler arguments from our earlier work
\cite{HSqm} that were used to prove the $\Omega$-bounds
implicit in \eqref{1.9}.  Instead, we shall have to
construct appropriate Knapp-type 
spectrally localized
quasimodes for 
these two different types of constant curvature geometries yielding the missing lower bounds.  In
the next section, we shall give the constructions
for compact space forms of positive curvature, while
in \S3, we shall use somewhat more involved arguments
to handle flat compact manifolds.  We shall use 
classical results about the structure of such space
forms that can be found, for instance, in
Charlap~\cite{flat} and Wolf~\cite{wolfconstant}.

We have defined here quasimodes as in \eqref{1.1} in terms of the spectrum of the functions involved.  Another common (but weaker)
definition is to require that
\begin{equation}\label{11.1}
\|\Phi\|_2+(\la \delta(\la))^{-1}\|(\Delta_g+\la^2)\Phi\|_2 \le 1.
\end{equation}
Clearly if $\Phi_\la\in V_{[\la,\la+\delta(\la)]}$ satisfies $\|\Phi_\la\|_2\le 1/4$, then $\Phi_\la$ satisfies the condition in 
\eqref{11.1}.  Also, it is a simple exercise (see, e.g. \cite{SZqm}) to see that the bounds in \eqref{1.4}, \eqref{1.6} and \eqref{1.7} yield
\begin{equation}\label{11.2}
\|\Phi\|_q\lesssim \la^{\mu(q)}\bigl(\, \|\Phi\|_2+\la^{-1}\|(\Delta_g+\la^2)\Phi\|_2\, \bigr), \, \,
q\in (2,q_c], \quad \text{any } \, M,
\end{equation}
\begin{multline}\label{11.3}
\|\Phi\|_q\lesssim (\la/\log\la)^{\mu(q)}\bigl(\, \|\Phi\|_2+ (\la/\log\la)^{-1} \|(\Delta_g+\la^2)\Phi\|_2\, \bigr), \, \,
q\in (2,q_c],
\\
\text{if all the sectional curvatures of } \, M \, \text{are nonpositive},
\end{multline}
and
\begin{multline}\label{11.4}
\|\Phi\|_q\lesssim \la^{\mu(q)}(\log\la)^{-1/2} \,  \bigl(\, \|\Phi\|_2+ (\la/\log\la)^{-1} \|(\Delta_g+\la^2)\Phi\|_2\, \bigr), \, \,
q\in (2,q_c],
\\
\text{if all the sectional curvatures of } \, M \, \text{are negative},
\end{multline}
respectively.

To use these bounds to obtain another classification of compact space forms, as we mentioned before, we shall use the
fact that another way of measuring concentration properties of quasimodes is using the potential decay rates of $L^2$-norms over shrinking
geodesic tubes.  See e.g., \cite{BSTop}, \cite{SoggeCon} and \cite{SoggeZelditchL4} for earlier results.  
To state our results let $\Pi$ be the space
of unit length geodesics in our compact manifold
$(M,g)$, and, for $\gamma \in \Pi$ and
$r\ll 1$, ${\mathcal T}_\gamma(r)$ a geodesic
$r$-tube about $\gamma$.
Using
\eqref{11.2}, \eqref{11.3}, \eqref{11.4} and the
Knapp-type constructions to follow we shall be able
to prove the following.

\begin{theorem}\label{tubethm}  Let 
\begin{equation}\label{11.5}
\widetilde V_\la=\{\Phi: \, \|\Phi\|_{L^2(M)}+(\la/\log\la)^{-1}\|(\Delta_g+\la^2)\Phi\|_{L^2(M)}\le 1\}.\end{equation}
Then for $\la\ge2$
if all the sectional curvatures of $M$ are nonpositive
\begin{equation}\label{1.12}
\sup_{\gamma\in \Pi, \, \Phi\in \widetilde V_\la} 
\|\Phi\|_{L^2({\mathcal T}_\gamma(R\la^{-1/2}))}=
O(R^{\frac{n-1}{n+1}}(\log\la)^{-\frac{n-1}{2(n+1)}}),
\, 1\le R\le (\log\la)^{1/2}.
\end{equation}
while if all of the sectional curvatures are negative
\begin{equation}\label{1.13} 
\sup_{\gamma\in \Pi, \, \Phi\in \widetilde V_\la} 
\|\Phi\|_{L^2({\mathcal T}_\gamma(\la^{-1/2}))}=
O((\log\la)^{-1/2}),
\end{equation}
and if $N\in {\mathbb N}$ is fixed
\begin{equation}\label{1.14}
\sup_{\gamma\in \Pi, \, \Phi\in \widetilde V_\la} 
\|\Phi\|_{L^2({\mathcal T}_\gamma((\log\la)^N\la^{-1/2}))}=
O_{\e, N}((\log\la)^{-1/2+\e}), \, \, 
\forall \, \e>0. 
\end{equation}
Furthermore, if $(M,g)$ is a compact space form
of curvature $K$ then
\begin{multline}\label{1.15}
\sup_{\gamma\in \Pi, \, \Phi\in \widetilde V_\la} 
\|\Phi\|_{L^2({\mathcal T}_\gamma(R\la^{-1/2}))}
\\
=
\begin{cases}
\Theta(1), \, \text{for }R=1 \quad \iff \,  K>0
\\
\Theta((R^{\frac{n-1}{n+1}}(\log\la)^{-\frac{n-1}{2(n+1)}}), \, \text{for }1\le R\le (\log\la)^{1/2} \quad
\iff \,  K=0
\\
O_{\e,N}((\log\la)^{-1/2+\e}), \, \, R=(\log\la)^N,
 \, 
\forall \e>0, \, 
N\in {\mathbb N}, \, \, 
\quad \iff \,  K<0.
\end{cases}
\end{multline}
\end{theorem}


For simplicity, we have only stated things here
for functions satisfying the quasimode condition \eqref{11.1} with $\delta(\la)=(\log\la)^{-1}$, however, there are analogous
results for  $\delta(\la)\searrow 0$ as above.

By arguing as in the proof of \eqref{1.9}, it is a
simple exercise using H\"older's inequality to see
that \eqref{11.3} and \eqref{11.4} yield \eqref{1.12},
\eqref{1.13} and \eqref{1.14}, respectively.  Also,
 trivially, the left side of \eqref{1.15}  is bounded by one.  As a result, in order to complete
the proof of Theorem~\ref{tubethm}, it suffices
to prove the $\Omega$-lower bounds implicit in
\eqref{1.15} for compact space forms of positive and
zero curvature.  As we shall see, the Knapp constructions
in the next two sections that will allow us to finish
the proof of Theorem~\ref{thm} will also do the same
for Theorem~\ref{tubethm}.

In the next two sections we shall complete the proofs
of our theorems.  Then, in the final section, we shall
state some problems related to our failure to obtain
$\Omega$-lower bounds in \eqref{1.11} and \eqref{1.15}
for manifolds all of whose sectional curvatures are negative.

\newsection{Knapp examples for $K>0$}

To finish the proof of Theorem~\ref{thm} and Theorem~\ref{tubethm} we need to prove the $\Omega$-lower bounds
implicit in \eqref{1.11} and \eqref{1.15} for compact space forms with sectional curvatures equal to $K>0$ as well
as for the flat case where $K=0$.  In this section we shall treat the positive curvature case.

To establish the lower bounds in \eqref{1.11} and \eqref{1.15} for $K>0$ we shall appeal to the following
proposition saying that there are highly focused Knapp examples for positively curved space forms.

\begin{proposition}\label{kp}  Suppose that $(M,g)$ is a connected compact space form with constant sectional curvature
$K>0$.  Then there is a periodic geodesic $\gamma_0\subset M$, a point $x_0\in \gamma_0$ and
a subsequence of eigenfunctions $e_{\la_{k_\ell}}$ with eigenvalues $\la_{k_\ell}\to \infty$ so that, for a uniform
constant $C_0<\infty$,
\begin{equation}\label{a}
\|e_{\la_{k_\ell}}\|_{L^2(M)}\le C_0 \quad
\text{and } \, \|e_{\la_{k_\ell}}\|_{L^1(M)}\le C_0\la^{-\frac{n-1}4}_{k_\ell},
\end{equation}
and, moreover,
\begin{equation}\label{b}
|e_{\la_{k_\ell}}(x)|\ge C_0^{-1}\la_{k_\ell}^{\frac{n-1}4}, \quad
x\in \mathcal{T}_{\gamma_0}(C^{-1}_0\la^{-1/2}_{k_\ell})\cap {\mathcal N}_0,
\end{equation}
with $\mathcal{N}_0$ being a fixed neighborhood of $x_0$.
\end{proposition}

Before proving the proposition, let us see how it leads to the aforementioned lower bounds.

For the ones needed for \eqref{1.11} with $K>0$, we note that since
$$|\mathcal{T}_{\gamma_0}(C^{-1}_0\la^{-1/2}_{k_\ell}) \cap {\mathcal N}_0|\approx
\la_{k_\ell}^{-\frac{n-1}2},
$$
by \eqref{b} we obtain
$$\|e_{\la_{k_\ell}}\|_{L^2(M)}
\ge \|e_{\la_{k_\ell}}\|_{L^2( \mathcal{T}_{\gamma_0}(C^{-1}_0\la^{-1/2}_{k_\ell}) \cap {\mathcal N}_0) } \ge c_0$$
for some uniform $c_0>0$.  Using this and the second part of \eqref{a} yields the uniform lower bounds
$$\|e_{\la_{k_\ell}}\|_{L^2(M)}/\|e_{\la_{k_\ell}}\|_{L^1(M)}\ge c_0'\la_{k_\ell}^{\frac{n-1}4}, \, \, 
c_0'>0,$$
which establishes the $\Omega$-bounds in \eqref{1.11} when $K>0$.

A similar argument yields the $\Omega$-bounds in \eqref{1.15} for $K>0$.

To prove the proposition we shall use the fact that the conclusions are valid for the special case
where $M$ is the round sphere $S^n$ of curvature $K=1$:

\begin{lemma}\label{kl}  Let
\begin{equation}\label{c}
Q_k(x)=k^{\frac{n-1}4}(x_1+ix_2)^k,
\, \, \, (x_1,x_2,x')\in S^n, \, \, \, x'=(x_3,\dots,x_{n+1}).
\end{equation}
Then
\begin{equation}\label{d}
\sqrt{-\Delta_{S^n}} \, Q_k=\la_k Q_k, \quad \la_k=\sqrt{k(k+(n-1)},
\end{equation}
and, moreover,
\begin{equation}\label{e}
\|Q_k\|_{L^2(S^n)}\approx 1 \quad \text{and } \, \,
\|Q_k\|_{L^1(S^n)}\approx k^{-\frac{n-1}4}\approx \la_k^{-\frac{n-1}4},
\end{equation}
as well as, for fixed $\delta>0$,
\begin{multline}\label{f}
|Q_k(x)|=O(k^{-\sigma}) \, \forall \, \sigma \in {\mathbb N} \, \, 
\, \text{if } x\in S^n \, \, \text{and }\, \, \text{dist }(x, \tilde \gamma_0)\ge \delta,
\\
\text{with }\, \, \tilde \gamma_0=\{(\cos\theta,\sin \theta, 0,\dots,0): \, \theta \in [0,2\pi)\}\subset S^n.
\end{multline}
Finally, if $\delta_0>0$ is small enough there is a uniform constant $c_0>0$ so that
\begin{equation}\label{f2}
|Q_k(x)|\ge c_0 k^{\frac{n-1}4} \quad \text{if } x\in S^n \, \, \text{and }\, \, \text{dist }(x,\tilde \gamma_0)\le \delta_0 k^{-1/2},
\end{equation}
\end{lemma}

This result is well known (see, e.g., \cite{SoggeHangzhou}).  The identity \eqref{d} follows from the fact
that $Q_k$  is a spherical harmonic of degree $k$, and \eqref{e}
-\eqref{f} follow from the fact that we can write
 \begin{equation}\label{qk}
 Q_k(x)=k^{\frac{n-1}4}\bigl(\sqrt{1-|x'|^2}\bigr)^k e^{ik\theta} \quad \text{if } \, \, 
 x\in S^n \, \, \text{and } \, \, \theta=\text{arg }(x_1+ix_2).
 \end{equation}

Let us now focus on the proof of Proposition~\ref{kp}.  Without loss of generality we may take $K=1$.  It then
follows that our space form $(M,g)$ is finitely covered by $S^n$.  Thus, we can identify all the eigenfunctions
on $M$ with eigenfunctions on $S^n$ (spherical harmonics) with certain symmetry properties.  The model
is real projective space $RP^n$ which has half the volume of $S^n$ and eigenfunctions corresponding to 
spherical harmonics of even order.  Recall that the distinct eigenvalues of $\sqrt{-\Delta_{S^n}}$ are
$\sqrt{k(k+n-1)}$, $k=0,1,2,\dots$, and, so, those on $RP^n$ are the above with $k=0,2,4,\dots$.  Also the 
multiplicity of each of these eigenvalues on $RP^n$ agrees with that on $S^n$, which is consistent with the fact that,
by the Weyl formula, we must have $N_{RP^n}(\la)=\tfrac12 N_{S^n}(\la) +o(N_{S^n}(\la) )$, by the above
volume considerations, if $N_M(\la)$ denotes the number of eigenvalues of $\sqrt{-\Delta_g}$ counted with
respect to multiplicity that are $\le \la$.

Returning to our compact space form of curvature $K=1$, since it is finitely covered by $S^n$, we can write
$M\simeq S^n/\Gamma$, where $\Gamma$ is finite group of isometries each of which is the restriction to
$S^n$ of an element $\Gamma \in O(n+1)$ (abusing notation a bit), with $O(n+1)$ being the 
orthogonal group for ${\mathbb R}^{n+1}$.  So, if $\tilde e_\la$ is an eigenfunction on $S^n$ with
eigenvalue $\la$, and, if $p:S^n\to M\simeq S^n/\Gamma$ is the covering map,
\begin{equation}\label{sum}e_\la(x)=\sum_{\alpha \in \Gamma}\tilde e_\la(\alpha(\tilde x)), \quad \text{if } \, x=p(\tilde x)\in M
\end{equation}
(locally) defines an eigenfunction on $M$ with the same eigenvalue $\la$ whenever $e_\la$
does not vanish identically..  In the model case, $RP^n$,
$\Gamma$ consists of the identity map $I$ and reflections about the origin.

If we take $\la_k$ as in \eqref{d} and
\begin{equation}\label{g}
\tilde e_{\la_k} =Q_k
\end{equation}
then clearly, by \eqref{e}, the resulting eigenfunctions on $M$ with eigenvalue $\la_k$ satisfy
the bounds in \eqref{a} with eigenvalue $\la_{k_\ell}$ being equal to $\la_k$.  So the proof
of Proposition~\ref{kp} would be complete if we could show that when for $\tilde \gamma_0\subset S^n$ as in
\eqref{f} and
\begin{equation}\label{h}
\gamma_0=p(\tilde \gamma_0)\subset M, 
\end{equation}
then
we have \eqref{b}.  Note that since $\tilde \gamma_0$ is a periodic geodesic (the equator) in $S^n$ of period $2\pi$,
its projection $\gamma_0$ is a periodic geodesic in $M$ (with perhaps a different period).

If $\alpha_0=Id$ and $\Gamma=\{\alpha_j\}_{j=0}^N \subset O(n+1)$ are the deck transformations above,
it follows that 
\begin{equation}\label{i}
\tilde \gamma_j=\alpha^{-1}_j(\tilde \gamma_0)\subset S^n, \quad j=1,2,\dots,N,
\end{equation}
like $\tilde \gamma_0$ 
are  great circles.  After possibly relabelling, we may assume that for some $m\in \{1,\dots,N\}$ we have
\begin{equation*}\tilde \gamma_j=\tilde \gamma_0, \, \, j\le m-1 \quad \text{but } \, \,
\tilde \gamma_j\ne \tilde \gamma_0 \quad \text{for } \, \, j\ge m.
\end{equation*}
We may further more assume that
\begin{equation}\label{j}
\one =(1,0,\dots,0)\notin \tilde \gamma_j, \quad j\ge m,
\end{equation}
and then we shall take the point $x_0\in M$ in Proposition~\ref{kp} to be $x_0=p(\one)$.

It follows that $\{\alpha_0,\dots,\alpha_{m-1}\}=\Gamma_0\subset \Gamma\subset O(n+1)$
are isometries not only mapping
 $S^n$ to itself but $\tilde \gamma_0$ into itself.  Thus, when acting on $\tilde \gamma_0$, $\Gamma_0$  is  isomorphic  to a cyclic group 
 $ {\mathbb Z}/m$ for some $m>1$.  So, after relabelling, we may assume that

  \begin{multline}\label{2.8}
 \alpha_j((\cos\theta,\sin\theta,0,\dots,0))=(\cos(\theta+\tfrac{2\pi j}m),
 \sin(\theta+\tfrac{2\pi j}m), 0,\dots,0), 
 \\
  j=0,\dots,m-1, \, \, \text{if } \, m>1.
 \end{multline}
 
 Note that by \eqref{f}, \eqref{g} and \eqref{j}, we can choose $\delta>0$
 \begin{equation}\label{2.9}
 \tilde e_{\la_k}(\alpha_j(\tilde x))=O(\la_k^{-\sigma}), \,  \, \forall \, \sigma, \, 
 \text{if } \, j\ge m, \, \, \text{and } \, \,
 \text{dist }(\tilde x, \one)< \delta.
 \end{equation}
 Thus, since $x_0=p(\one)$, if $\mathcal{N}_0$ is a $\delta$-ball about $x_0$, the summands in \eqref{sum} corresponding
 to $j\ge m$ are trivial here:
 \begin{equation}\label{2.10}\sum_{j\ge m} |\tilde e_{\lambda_k}(\alpha_j(\tilde x))|=O(\la_j^{-\sigma}) \, \forall \, \sigma  \quad \text{if } \, \,
 \text{dist }(\tilde x, \one)< \delta.
 \end{equation}

Thus, if $m=1$, for all $\sigma$ we have
\begin{equation}\label{2.11}
\sum_{\alpha\in \Gamma} \tilde e_{\la_k}(\alpha(\tilde x))
=Q_k(\tilde x)
+O(\la_k^{-\sigma}),
\quad \text{if } \, \,
 \text{dist }(\tilde x, \one)< \delta.
\end{equation}
From this  and \eqref{f2} we deduce that if ${\mathcal N}_0$ is above then we must have \eqref{b} if $C_0$ is large enough
if $m=1$ and $\la_{k_\ell}=\la_\ell$, $\ell =1,2,\dots$.  So, when $m=1$, we do not have to pass to a subsequence of
eigenvalues.

If $m>1$ we do have to pass to a subsequence to ensure that there is no cancellation over the sum corresponding to the
cyclic subgroup $\Gamma_0=\{\alpha_0,\dots, \alpha_{m-1}\}$ for which we have \eqref{2.8}.  Since $\alpha_j\in
O(n+1)$, $j\le m-1$, rotates the
2-plane $x'=0$ in ${\mathbb R}^{n+1}$ by angle $2\pi j/m$, it follows that there must be $m_j\in O(n-1)$ so that
$$\alpha_j=
\begin{pmatrix}
R_{2\pi j/m} & 0 \\
0 &m_j \end{pmatrix},
$$
with $R_{2\pi j/m} \in O(2)$ denoting rotating by this angle.   Since $|m_j y'|=|y'|$ it follows from \eqref{qk} that we must have
for each $j\le m-1$ and $\ell \in {\mathbb N}$
\begin{multline*}
Q_{\ell m}(\alpha_j(\tilde x))=(\ell m)^{\frac{n-1}4}\bigl(\sqrt{1-|m_j \tilde x'|^2}\bigr)^{\ell m} e^{i \ell m (\theta + 2\pi j/\ell m)}
\\
=(\ell m)^{\frac{n-1}4} \bigl(\sqrt{1-|\tilde x'|^2}\bigr)^{\ell m} e^{ilm \theta}
=Q_{k_\ell}(\tilde x), \quad k_\ell =\ell m.
\end{multline*}
Thus, if $m>1$, we have the following variant of \eqref{2.11}
\begin{equation}\label{2.11'}
\sum_{\alpha\in \Gamma} \tilde e_{\la_k}(\alpha(\tilde x))
= m Q_k(\tilde x)
+O(\la_k^{-\sigma}),\quad \text{if } \, \,
 \text{dist }(\tilde x, \one)< \delta,
\end{equation}
which gives us \eqref{b} just as before.  \qed


  \newsection{Knapp examples for $K=0$} 
  
   In this section we shall prove the $\Omega$-lower bounds
implicit in \eqref{1.11} and \eqref{1.15} for compact space forms with sectional curvatures equal to $K=0$. By arguing as in the remarks below Proposition~\ref{kp}, it suffices to prove the following

\begin{proposition}\label{kpt}  Suppose that $(M,g)$ is a connected compact space form with constant sectional curvature
$K=0$.  Then there is a periodic geodesic $\gamma_0\subset M$, a point $x_0\in \gamma_0$ and
a sequence of quasimodes $\psi_{\la_{k}}$ with $\psi_{\la_{k}}\in V_{[\la_k-\delta_k/2,\la_k+\delta_k/2]}$,  $\la_{k}\to \infty$ and $\delta_{k}\in [(\log\lambda_{k})^{-1},1] $ so that, for a uniform
constant $C_0<\infty$,
\begin{equation}\label{at}
\|\psi_{\la_{k}}\|_{L^2(M)}\le C_0 \quad
\text{and } \, \|\psi_{\la_{k}}\|_{L^1(M)}\le C_0(\la_{k}\delta_{k})^{-\frac{n-1}4},
\end{equation}
and, moreover,
\begin{equation}\label{bt}
|\psi_{\la_{k}}(x)|\ge C_0^{-1}(\la_{k}\delta_{k})^{\frac{n-1}4}, \quad
x\in \mathcal{T}_{\gamma_0}(C^{-1}_0(\la_{k}\delta_{k})^{-1/2})\cap {\mathcal N}_0,
\end{equation}
with $\mathcal{N}_0$ being a fixed neighborhood of $x_0$.
\end{proposition}
Here for $K=0$ the assumption $\delta_{k}\in [(\log\lambda_{k})^{-1},1] $ can be weakened to $\delta_{k}\in [\lambda_{k}^{-1+\e},1]$ for any $\e>0$ without changing the proof.

By a classical theorem of Cartan and Hadamard, if $(M,g)$ is a compact flat manifold, it must be of the form
$\Rn/ \Gamma$, where $\Gamma$ is the set of deck transformations. Also by a theorem of 
Bierbach \cite{bieberbach} (see e.g., Corollary 5.1 and 
Theorem 5.3 in Chapter 2 in \cite{flat}
or \cite{wolfconstant}),  $\Gamma$ must be a Bieberbach subgroup of the group rigid motions, $E(n)$, of $\Rn$. In addition, if we  fix an arbitrary point $x_0\in M$ and let $p=\exp_{x_0}: \Rn\to M$, then $p$ is a covering map.  If $D\subset \Rn$ is a Dirichlet domain containing the origin,
then we can identify $D$ with $M$ by setting $p(\tilde x)=x$ if $\tilde x\in D$.

A function $f$ in $\R^n$ is called periodic in $\Gamma$ if $f(x)=f(\alpha(x))$ for all $x\in \R^n$ and $\alpha\in \Gamma$. 
  For a given $\tilde f\in C^\infty(\R^n)$ which is periodic in $\Gamma$, we can define
 $f(x)=\tilde f(\tilde x)$ if $p(\tilde x)=x$, and if $\Delta=\partial^2/\partial x_1^2+\cdots + \partial^2/\partial x_n^2$ is the standard
Laplacian we have
\begin{equation}\label{a4.1}
\Delta_g f(x)=\Delta \tilde f(\tilde x)
\end{equation}
Similarly, for a given $ f\in C^\infty(M)$, we can define  $\tilde f(\tilde x)=f(x)$ if $p(\tilde x)=x$ for $\tilde x\in D$, and extend $\tilde f$ to a smooth periodic function in $\R^n$ which satisfies \eqref{a4.1}.

Furthermore,
by the first part of Bieberbach's theorem (see e.g.,  Chapter 2 in \cite{flat}), if we denote $\Lambda\subset \Gamma$ to be the subgroup of translations, then $\Lambda$ has finite index. In other words, $\Lambda$ consists of translations by a lattice of full rank, and $\mathbb{T}^n\simeq\R^n/\Lambda$ is a flat torus, which may not be the standard torus depending on the lattice.  And as discussed above, any function defined on 
$\mathbb{T}^n$ can be identified naturally with a function in $\R^n$ which is periodic in $\Lambda$.

We shall need the following lemma about relations between periodic functions in $\Gamma$ and $\Lambda$, 
\begin{lemma}\label{alemma}
There exist finitely many $\alpha_i\in \Gamma$, which satisfy
\begin{equation}\label{alphacondition}
    \alpha_1=\text{Identity},\,\, \alpha_i\notin \Lambda  \,\,\,  i=2,3,\dots,N,
\end{equation}
 such that   if  $\tilde f$ is a function in $\R^n$ which is periodic in $\Lambda$
    \begin{equation}\label{fper}
     f(x)=\sum_{i=1}^N \tilde f(\alpha_i(x)) 
    \end{equation}
    is a periodic function in $\Gamma$.
\end{lemma}

As an example, in the case of Klein bottle, $N=2$ and we can take $\alpha_2$ to be any element in $\Gamma$ that is not a  translation.

\begin{proof}
    Assume the subgroup of translations $\Lambda$ has index $N$, then we can write $\Gamma$ as  $\Gamma=\cup_{i=1}^N \Lambda \alpha_i $, where $\Lambda \alpha_i=\{g\circ\alpha_i: g\in \Lambda \} $ are mutually disjoint right cosets of $\Lambda$.  If we take $\alpha_1=\text{Identity}$, then it is clear that $\alpha_i\notin \Lambda$ for $i\ge 2$ since 
$\Lambda \alpha_i$ are mutually disjoint.

It remains to show that for the choice of $\alpha_i$ above, $f(x)$ defined in \eqref{fper} is periodic in $\Gamma$. To see this, we claim that, for any $\alpha\in \Gamma$, $\alpha_i\alpha=g_i\alpha_{j(i)}$ for some $g_i\in \Lambda$ and $j(i)\in \{1, 2,\dots N\}$, and furthermore, $j(i_1)\neq j(i_2)$ if $i_1\neq i_2$. The first fact just follows from $\Gamma=\cup_{i=1}^N \Lambda \alpha_i $. To show that $j(i)$ is injective, assume for $i_1\neq i_2$, there exist a $1\le j\le N$  such that $\alpha_{i_1}\alpha=g_{i_1}\alpha_{j}$ and $\alpha_{i_2}\alpha=g_{i_2}\alpha_{j}$, this implies that $\alpha_{i_2}=g_{i_2}g^{-1}_{i_1}\alpha_{i_1}$, which would lead to the contradiction that $\Lambda \alpha_{i_1}=\Lambda \alpha_{i_2}$, thus the claim follows. Using this along with the fact that $\tilde f(x)$ is periodic in $\Lambda$, we have for any $\alpha\in\Gamma$
\begin{equation}
    \begin{aligned}
        f(\alpha(x))&=\sum_{i=1}^N \tilde f(\alpha_i\alpha(x))=\sum_{i=1}^N \tilde f(g_i\alpha_{j(i)}(x))\\
        &=\sum_{i=1}^N \tilde f(\alpha_{j(i)}(x))=f(x)
    \end{aligned}
\end{equation}\end{proof}

The proof of Proposition~\ref{kpt} follows the same lines as that of Proposition~\ref{kp} for $K>0$.  So, we shall first construct quasimodes that satisfy the desired concentration properties on $\mathbb{T}^n$, and then use Lemma~\ref{alemma} to build up associated quasimodes on the original manifold $M$.


To proceed, let us assume ${b}_1, {b}_2, \dots, {b}_n$ are linearly independent vectors in $\R^{n\times 1}$, and 
\begin{equation}\label{3.7}
    \Lambda=\{{Bx}: {x}\in {\mathbb Z}^n\}=\{\sum_{i=1}^n x_i{b}_i:x_i\in {\mathbb Z}\},
\end{equation}
where ${B}=[{b}_1, {b}_2, \dots, {b}_n]$ is an $n\times n$ invertible matrix. For simplicity, let us first assume ${b}_n=(0, 0, \dots, 0, s)^{T}$, we shall discuss how to modify the arguments for general $b_n$ at the end of this section by using a rotation matrix. 

Any periodic functions in $\Lambda$ must be of form 
\begin{equation}
    \sum_{\xi\in \mathbb Z^n} a_{\xi}e^{2\pi i \xi \cdot{B}^{-1}x},
\end{equation}
where each $e^{2\pi i \xi\cdot {B}^{-1}x}$ is a periodic function in $\Lambda$, which corresponds to an eigenfunction of $\Delta$ on $\mathbb{T}^n$
with eigenvalue  $| (B^T)^{-1}\xi|^2$. By using a change of variable $x={B}y$, each $e^{2\pi i \xi  \cdot {B}^{-1}x}$ can also be identified with $e^{2\pi i \xi\cdot y}$, which is an eigenfunction of $Q(\nabla)$ on $\R^n/\mathbb{Z}^n$ with the same eigenvalue, if $Q(\xi)=\sum_{i,j}\beta_{ij}\xi_i\xi_j$ is a positive definite quadratic form in $\R^n$ with $\beta_{ij}$ being the $i,j$th element of the matrix 
${(B^T)^{-1}B^{-1}}$.

In \cite{Germain}, Germain and Myerson constructed the following example on $\R^n/\mathbb{Z}^n$
\begin{equation}\label{example}
    f(y)=e^{2\pi i\la(\xi_0)_n\cdot y_n}\sum_{\xi'\in \mathbb{Z}^{n-1}}\phi(\frac{\xi'-\lambda\xi_0'}{\sqrt{\lambda\delta}})e^{2\pi i\xi'\cdot y'},
\end{equation}
where $\phi\in C_0^\infty(\R^{n-1})$ is such that   $\widehat\phi(x)\ge 1$ for $|x|\le 1$, $y=(y', y_n)$, and $\xi_0=(\xi_0', (\xi_0)_n)$ is a point on the ellipse $\{\xi\in\R^n, Q(\xi)=1\}$ such that the normal vector to the ellipse at $\xi_0$ is colinear to $(0,0,\dots,0,1)$ and $\lambda(\xi_0)_n\in \mathbb{N}$. Due to  the choice of  $\xi_0$ , it is not hard to check that if we choose the support of $\phi$ to be small enough centered around origin, we have 
\begin{equation}\label{spec}
   \lambda-\delta/2\le  \sqrt{Q(\xi)}\le \lambda+\delta/2, \,\,\,\text{if} \,\,\xi=(\xi',\la(\xi_0)_n)\,\, \,\text{and}\,\,\,\phi(\frac{\xi'-\lambda\xi_0'}{\sqrt{\lambda\delta}})\neq 0.
\end{equation}

By the Poisson summation formula, $f$ can also be written as 
\begin{equation}\label{per}
    e^{2\pi i\la(\xi_0)_n\cdot y_n}(\lambda\delta)^{\frac{n-1}{2}}\sum_{\eta\in\mathbb{Z}^{n-1}}\widehat\phi(\sqrt{\lambda\delta}(y'-\eta)) e^{2\pi i\la\xi_0'\cdot (y'-\eta)}.
\end{equation}
Using \eqref{per}, it is straightforward calculation to check that if $D=[0,1]^n$, or any of its translations in $\R^n$,
\begin{equation}\label{size}
    \|f\|_{L^p(D)}\sim (\lambda\delta)^{\frac{n-1}{2}(1-\frac1p)},\,\,\text{if}\,\,\,\lambda\delta\gg1.
\end{equation}

To prove Proposition~\ref{kpt}, let us fix $\la=\la_k, \delta=\delta_k$ in \eqref{example} and define
\begin{equation}\label{example1}
\begin{aligned}
      \varphi_{\la_k}(y)&=(\lambda_k\delta_k)^{-\frac{n-1}{4}} f(y-y_0)\\
      &=e^{2\pi i\la_k(\xi_0)_n\cdot y_n}(\lambda_k\delta_k)^{-\frac{n-1}{4}}\sum_{\xi'\in \mathbb{Z}^{n-1}}\phi(\frac{\xi'-\lambda_k\xi_0'}{\sqrt{\lambda_k\delta_k}})e^{2\pi i\la_k\xi'\cdot (y'-y_0)} \\
       &= e^{2\pi i\la_k(\xi_0)_n\cdot y_n}(\lambda_k\delta_k)^{\frac{n-1}{4}}\sum_{\eta\in\mathbb{Z}^{n-1}}\widehat\phi(\sqrt{\lambda_k\delta_k}(y'-y'_0-\eta)) e^{2\pi i\la_k\xi_0'\cdot (y'-y'_0-\eta)},
\end{aligned}
\end{equation}
where $\lambda_k(\xi_0)_n=k$ for $k\in \mathbb{Z}$ with $k\ge C$ for some $C$ large enough, and $y_0=(y_0', 0)$ with $|y_0'|= c_0$ for some small constant $c_0$. 
The choice of $y_0$ may depend on $\alpha_i$, but not on $\la_k$;
we shall specify the details on the choice of $y_0$ later.

By using the changes of variable $y={B}^{-1}x$ and Lemma~\ref{alemma}, it is clear that 
\begin{equation}
    \tilde \psi_{\la_k}(x)= \sum_{i=1}^N\varphi_{\la_k}({B}^{-1}\alpha_i(x))
\end{equation}
defines a smooth periodic function in $\Gamma$. And by \eqref{spec}, if we identify $ \tilde \psi_{\la_k}(x)$ with a function $\psi_{\lambda_k}\in C^\infty (M)$  via the covering map, we have $\text{Spec }\psi_{\la_k} \subset [\la_k-\delta_k/2,\la_k+\delta_k/2]$. Next,  let $D\subset \Rn$ be a Dirichlet domain for $M$ containing the origin, and
\begin{equation}\label{tube}
    \mathcal{T}_{k,x_0}=\{(x', x_n)\in\R^n: |x'-x_0'|\le c_1 (\la_k\delta_k)^{-1/2},\,\, |x_n-(x_0)_n|\le c_2\},
\end{equation}
where $x_0'=A^{-1}y_0'$ for $A$ defined in \eqref{4.2} below, and $c_1, c_2, (x_0)_n>0$ are small constants to be specified later which are independent of $\la_k, \delta_k$.  Since we are assuming ${b}_n=(0, 0, \dots, 0, s)^{T}$, by \eqref{3.7},  the group $\Gamma$ contains the deck transformation $\alpha$ with $\alpha(x)=(x', x_n+s)$, thus the straight line 
$\{(x_0', t): t\in \R\}$ must be mapped to some 
periodic geodesic $\gamma_0\in M$ by the covering map $p$, and 
$\mathcal{T}_{k,x_0}$ must be mapped onto  a $c_1(\la_{k}\delta_{k})^{-\frac12}$ neighborhood of  $\gamma_0$ near the point $p(x_0)$ if $x_0=(x_0', \,\,(x_0)_n)$.

Thus, to prove Proposition~\ref{kpt}, it suffices to show  
\begin{equation}\label{at1}
\|\tilde\psi_{\la_{k}}\|_{L^2(D)}\le C_0 \quad
\text{and } \, \|\tilde\psi_{\la_{k}}\|_{L^1(D)}\le C_0(\la_{k}\delta_{k})^{-\frac{n-1}4},
\end{equation}
as well as 
\begin{equation}\label{bt1}
|\tilde\psi_{\la_{k}}(x)|\ge C_0^{-1}(\la_{k}\delta_{k})^{\frac{n-1}4}, \quad
x\in \mathcal{T}_{k,x_0}.
\end{equation}

The proof of \eqref{at1} just follows from \eqref{size} and the definition of $\varphi_{\la_k}$ in \eqref{example1}. To prove \eqref{bt1}, let us first make some simple reductions.
Recall that since ${b}_n=(0, 0, \dots, 0, s)^{T}$, we have 
\begin{equation}\label{4.2}
B^{-1}=    \left(
    \begin{array}{ccccc}
     &  & 0 \\ 
     \quad \mbox{ \huge${A}$}& & \vdots  
\\
 & & 0\\
a_1 \,\,\,  \cdots \,\,\,a_{n-1}& & s^{-1}
    \end{array}
    \right),
\quad \text{with } \,   {A}\in GL_{n-1}(\R).
\end{equation}
Also, recall that $\alpha_i\in E(n) $, the group of rigid motions, and so we may assume $\alpha_i(x)=m_ix+j_i$, where $m_i\in O(n)$ is an $n\times n$ orthogonal matrix and $j_i=(j_i',(j_i)_n)\in \R^n$. Here, since $\varphi_{\la_{k}}$ is periodic in $\mathbb{Z}^n$ and $B^{-1}\alpha_i(x)=(B^{-1}m_i)x+B^{-1}j_i$, without loss of generality, we may assume $B^{-1}j_i\in [0,1)^n$, 
and furthermore, since the number of choices of $i$ is finite, we can assume
$\alpha_i(x)=m_ix+j_i\,\,\text{with}\,\,B^{-1}j_i\in [0,1-\delta_0)^n
$
 for some $0<\delta_0<1$ uniformly in $i$. Thus, by choosing the constants in \eqref{tube} small enough so that $\mathcal{T}_{k,x_0}$ is close enough to the origin and $|y'_0|=c_0\ll \delta_0$, we may assume that 
\begin{equation}\label{small}
    B^{-1}\alpha_i(x)-y_0\in [-1+\delta_0/2,1-\delta_0/2)^n,\,\,\,i=1,2\cdots, N, \quad \,\,\text{if}\,\,\,
x\in \mathcal{T}_{k,x_0}.
\end{equation}
If we let $\pi:\pi(x)=x'$ be the projection map onto the first $n-1$ coordinates, \eqref{small} implies that 
\begin{equation}
  \left|  \widehat\phi(\sqrt{\lambda_k\delta_k}(A\pi\alpha_i(x)-y'_0-\eta))\right|\le C_\sigma (\lambda_k\delta_k|\eta|)^{-\sigma}\,\, \forall \, \sigma \in {\mathbb N},\,\, \eta\in\mathbb{Z}^{n-1},\,\,|\eta|\neq 0.
\end{equation}
Consequently, if $\la_k\delta_k\to \infty$ as  $k\to \infty$ and  $k\ge C$ is large enough, to prove \eqref{bt1}, it suffices to assume  $\eta=0$, and show that 
\begin{equation}\label{bt22}
|\sum_{i=1}^N \widehat\phi(\sqrt{\lambda_k\delta_k}(A\pi\alpha_i(x)-y'_0))|\ge (2C_0)^{-1} \,\,\,\,\text{if}\,\,
x\in \mathcal{T}_{k,x_0}.
\end{equation}

To proceed, let us define $\ell_{0}=\{(x_0', t): t\in \R\}$, where $x_0'=A^{-1}y_0'$, then since $\widehat\phi\in \mathcal{S}$ and $\widehat\phi(x)\ge 1$ for $|x|\le 1$, it is straightforward to check that for fixed $\delta>0$
\begin{equation}\label{decay}
    |\widehat\phi(\sqrt{\lambda_k\delta_k}(A\pi(x)-y'_0))|\le C_\sigma (\lambda_k\delta_k)^{-\sigma}\,\, \forall \, \sigma \in {\mathbb N}, \, \, \,\text{if}\,\, \, \text{dist }(x, \ell_0)\ge \delta,
\end{equation}
and for some  $c_1>0$ small enough which may depend on $A$,
\begin{equation}\label{ff2}
 |\widehat\phi(\sqrt{\lambda_k\delta_k}(A\pi(x)-y'_0))|\ge 1  \quad \text{if } \text{dist }(x, \ell_0)\le c_1 (\la_k\delta_k)^{-1/2}.
\end{equation}

By the definition of $\mathcal{T}_{k,x_0}$ in \eqref{tube}, we certainly have $ \text{dist }(x, \ell_0)\le c_1 (\la_k\delta_k)^{-1/2}$ if $x\in \mathcal{T}_{k,x_0}$. 
Thus, to prove \eqref{bt22}, it suffices to show for $x\in \mathcal{T}_{k,x_0}$
\begin{equation}\label{bt23}
|\sum_{i=2}^N \widehat\phi(\sqrt{\lambda_k\delta_k}(A\pi\alpha_i(x)-y'_0))|\le  C_\sigma (\lambda_k\delta_k)^{-\sigma}\,\, \forall \, \sigma \in {\mathbb N}.
\end{equation}

If for each $\alpha\in \Gamma$,  we also define $\ell_{\alpha}=\{x: \alpha(x)\in \ell_0\}$, the preimage of $\ell_0$ under $\alpha$, then we claim that \eqref{bt23} would be a consequence of the following lemma.

\begin{lemma}\label{nointersect}
    There exist a $x_0'\in\R^{n-1}$, where $|x_0'|$ can be arbitrary small, such that if 
    $\ell_{0}=\{(x_0', t): t\in \R\}$ ,    $\ell_{\alpha_i}\neq \ell_0$ for all $i=2,\dots, N$.
\end{lemma}
Let us first prove the claim. Note that for each $i$, \eqref{decay} is equivalent to
\begin{equation}
    |\widehat\phi(\sqrt{\lambda_k\delta_k}(A(\pi\alpha_i(x))-y'_0))|\le C_\sigma (\lambda_k\delta_k)^{-\sigma}\,\, \forall \, \sigma \in {\mathbb N}, \, \, \,\text{if}\,\, \, \text{dist }(x, \ell_{\alpha_i})\ge \delta.
\end{equation}
Thus, we would have \eqref{bt23} if we can choose $\mathcal{T}_{k,x_0}$ such that for some fixed $\delta>0$,
\begin{equation}\label{far}
    \text{dist }(x, \ell_{\alpha_i})\ge \delta,  \,\,i=2,3,\dots, N\quad \text{if}\,\,
x\in \mathcal{T}_{k,x_0}.
\end{equation}

To prove this, by Lemma~\ref{nointersect}, each line $\ell_{\alpha_i}$ can intersect $\ell_0$ at most once.   Thus, there must be a point $x_0=(x_0', (x_0)_n)$ with $|x_0|\le c$ such that $x_0\notin \ell_{\alpha_i}$ for all $i\ge 2$. Here we emphasize that the constant $c$ can be chosen  small enough, which is necessary for \eqref{small} to hold. Since $x_0\in \mathcal{T}_{k,x_0}$, if we define $\delta_1=\min_i\, \dist(x_0, \ell_{\alpha_i})$, by further choosing the constants $c_1, c_2$ in \eqref{tube} to be small enough relative to $\delta_1$, for $\la_k\delta_k\gg \delta_1^{-1}$, which holds as long as $k\ge C_0$ is large enough, we must have \eqref{far} with $\delta=\delta_1/2$,  which finishes the proof of \eqref{bt23}.
\begin{proof}[Proof of Lemma~\ref{nointersect}]
Note that $\ell_{\alpha_i}\neq \ell_0$ is equivalent to $\alpha_i(\ell_0)\neq \ell_0$, if $\alpha_i(\ell_0)$ denotes the image of $\ell_0$ under $\alpha_i$. For $\alpha_i(x)=m_ix+j_i$ with $j_i=(j_i', (j_i)_n)$, we shall divide our discussion into the following cases.

(i) $j_i'\neq 0$. In this case,
since the number of choices of $i$ is finite, we may assume $|j_i'|\ge \delta_2$ for some constant $\delta_2$ independent of $i$, thus if we choose $x_0'$ such that $|x_0'|\le \frac12\delta_2$, it is not hard to check that for $(x_0', 0)\in \ell_0$,  $\alpha_i((x_0', 0))\notin \ell_0$ for all $i$ satisfying (i).

(ii)  $m_i (0, 0, \dots, 0, 1)^{T}\neq (0, 0, \dots, 0, \pm 1)^{T}$.  In this case, the directions of  $\alpha_i(\ell_0)$ and $\ell_0$ are transverse, thus
we certainly have $\alpha_i(\ell_0)\neq \ell_0$ for all $i$  satisfying (ii).

(iii) $j_i'= 0$ and $m_i (0, 0, \dots, 0, 1)^{T}= (0, 0, \dots, 0, \pm 1)^{T}$. In this case, we may assume
$\alpha_i(x)=m_ix+(0,\cdots,0,  (j_i)_n)$, where
\begin{equation}\label{4.3}
m_i=    \left(
    \begin{array}{ccccc}
     &  & 0 \\ 
     \quad \mbox{ \huge$\overline{m}_i$}& & \vdots  
\\
 & & 0\\
0 \,\,\,  \cdots \,\,\,0& & \mu_i
    \end{array}
    \right),
\quad \text{with } \,   \overline{m}_i\in O(n-1),\,\, \mu_i=\pm 1.
\end{equation}
We claim that $\overline{m}_i\neq I_{n-1}$. To see this, note that  if $\overline{m}_i= I_{n-1}$ and $\mu_i=1$,  we have $\alpha_i(x)=x+(0,\cdots, (j_i)_n)$, which is impossible since by \eqref{alphacondition}, $\alpha_i$ is not a translation for $i\ge 2$.  On the other hand, 
if $\overline{m}_i= I_{n-1}$ and $\mu_i=-1$, we have $\alpha_i(x)=x$ for $x=(0,\cdots,0, \frac{(j_i)_n}{2})$, which contradicts with the fact that the deck group $\Gamma$ acts freely on the covering space $\R^n$.

Since $\overline{m}_i\neq I_{n-1}$, the set of fixed points  $\{x'\in \R^{n-1}:\overline{m}_ix'=x'\}$ must be a subset of some hyperplane containing origin. Thus, there must exist a point $x'_0\in \R^{n-1}$ that is outside these finitely many hyperplanes, i.e., $\overline{m}_ix_0'\neq x_0'$ for all values of $i$ satisfying (iii), which implies $\alpha_i(\ell_0)\neq \ell_0$. Also, since $\overline{m}_i$ is a linear matrix, the same results hold if we replace $x_0'$ by the point $cx_0'$ for any $c<1$. 
\end{proof}
Finally, we shall deal with the general case $b_n\neq (0,\dots, 0,s)^T$. Let us choose an orthogonal matrix $Q$ such that $Q b_n=(0,\dots, 0,s)^T$ for some $s$, and write $B=Q^TQB=Q^T\bar B$, where the last column of $\bar B$ is $(0,\dots, 0,s)^T$ . In this case, we can define
\begin{equation}
    \tilde \psi_{\la_k}(x)= \sum_{i=1}^N\varphi_{\la_k}({B}^{-1}\alpha_i(x))=\sum_{i=1}^N\varphi_{\la_k}({\bar B}^{-1}\bar{\alpha}_i(Qx)),
\end{equation}
where $\bar{\alpha}_i=Q{\alpha}_iQ^T$. 
Then $\bar B^{-1}$ satisfies \eqref{4.2}, and $\bar \alpha_i\in E(n)$  satisfies the two properties we used for $\alpha_i$, i.e., $\bar \alpha_i$ is not a translation and does not have a fixed point. Thus we can define $\bar x=Qx$ and repeat the above arguments to prove Proposition~\ref{kpt} for the general case.


\newsection{Problems and comparison with earlier results}

We stated Theorem~\ref{tubethm} in terms of the log-quasimodes defined by \eqref{11.5}, since these were the ones considered by
Brooks~\cite{BrooksQM} for compact space forms of negative sectional curvature.  On the other hand, all of the conclusions
remain valid if we consider log-quasimodes defined by the spectral condition \eqref{1.1}.  Specifically, we have the following.

\begin{theorem}\label{thm4.1}  Consider
$\{\Phi_\la\in V_{[\la,\la+(\log\la)^{-1}]}: \,
\|\Phi_\la\|_2=1\}$.  Then for $\la\ge2$
if all the sectional curvatures of $M$ are nonpositive
\begin{equation}\label{4.1}
\sup_{\gamma\in \Pi} 
\|\Phi_\la\|_{L^2({\mathcal T}_\gamma(R\la^{-1/2}))}=
O(R^{\frac{n-1}{n+1}}(\log\la)^{-\frac{n-1}{2(n+1)}}),
\, 1\le R\le (\log\la)^{1/2}.
\end{equation}
while, if all of the sectional curvatures are negative,
\begin{equation}\label{4.2} 
\sup_{\gamma\in \Pi} 
\|\Phi_\la\|_{L^2({\mathcal T}_\gamma(\la^{-1/2}))}=
O((\log\la)^{-1/2}),
\end{equation}
and $N\in {\mathbb N}$ is fixed
\begin{equation}\label{4.3}
\sup_{\gamma\in \Pi} 
\|\Phi_\la\|_{L^2({\mathcal T}_\gamma((\log\la)^N\la^{-1/2}))}=
O_{\e, N}((\log\la)^{-1/2+\e}), \, \, 
\forall \, \e>0. 
\end{equation}
Furthermore, if $(M,g)$ is a compact space form
of curvature $K$ then
\begin{multline}\label{4.4}
\sup_{\gamma\in \Pi} 
\|\Phi_\la\|_{L^2({\mathcal T}_\gamma(R\la^{-1/2}))}
\\
=
\begin{cases}
\Theta(1), \, \text{for }R=1 \quad \iff \,  K>0
\\
\Theta((R^{\frac{n-1}{n+1}}(\log\la)^{-\frac{n-1}{2(n+1)}}), \, \text{for }1\le R\le (\log\la)^{1/2} \quad
\iff \,  K=0
\\
O_{\e,N}((\log\la)^{-1/2+\e}), \, \, R=(\log\la)^N,
 \, 
\forall \e>0, \, 
N\in {\mathbb N}, \, \, 
\quad \iff \,  K<0.
\end{cases}
\end{multline}
\end{theorem}

The proof of this result is almost identical to that of Theorem~\ref{tubethm}.  The upper bounds \eqref{4.1}, \eqref{4.2} and \eqref{4.3}, just
as before, follow via H\"older's inequality from \eqref{1.4}, \eqref{1.6} and \eqref{1.7}.  Moreover, since the quasimodes in Propositions~\ref{kp}
and \ref{kpt} are spectrally localized, we can repeat the earlier arguments to obtain the $\Omega$-lower bounds implicit in the first two parts
of \eqref{4.4}, which completes the proof.

It would be interesting to try to obtain the analog of Theorem~\ref{thm} for the log-quasimodes defined by \eqref{11.5}.  We are unable to prove the
upper bounds \eqref{1.10} for $\Phi \in \widetilde V_\la$ since, unlike \eqref{1.4}, \eqref{1.6} and \eqref{1.7}, the bounds in \eqref{11.1} for 
$\widetilde V_\la$ do not only involve $L^2$-norms in the left.  As a result, we can not use H\"older's inequality as we did for the proof of the 
upper bounds in Theorem~\ref{thm}.

It would also be interesting to try to replace the $O$-upper bounds in Theorem~\ref{thm}, \ref{tubethm} or \ref{thm4.1} with $\Theta$-bounds
for compact space forms of negative sectional curvature (i.e., $K<0$).  This appears difficult.  One thing that is missing for this case is analog of
Propositions~\ref{kp} or \ref{kpt} for compact space forms of negative curvature.  We could obtain the Knapp examples in these propositions
since there were model modes on  covers of positively curved or flat compact space forms, i.e., the round sphere or tori, respectively.  We are unaware of  any simple model modes for compact space forms of negative curvature.

It would also be interesting to try to bridge the gap between the concentration results obtained here and in Brooks~\cite{BrooksQM}.  Brooks' results imply that if $\gamma_0$ is a periodic geodesic in a compact space form of negative curvature, then there is a $\delta>0$ and a 
sequence of frequencies $\la_j\to \infty$  and associated log-quasimodes $\Phi=\Phi_\la$ as in \eqref{11.5}
so that if ${\mathcal N}$ is any {\em fixed} neighborhood of $\gamma_0$ then $\liminf \|\Phi_\la\|_{L^2({\mathcal N})}>\delta$.  On the other
hand, by Theorem~\ref{tubethm},  for any fixed $N\in {\mathbb N}$, if
${\mathcal N}={\mathcal N}(\la)={\mathcal T}_{\gamma_0}((\log\la)^N\la^{-1/2})$, then we must have
$\liminf \|\Phi_\la\|_{L^2({\mathcal N(\la)})}=0$.    These tubes have volume tending to zero; however, are there analogs of the lower bounds
of Brooks for tubes of volume tending to zero slower than the ones above?  Also, for the case of compact space forms with $K>0$, as we have
shown, tube radius $\la^{-1/2}$ is a natural barrier for lower bounds of
$L^2$-norms over geodesic tubes, while for flat manifolds, as in \eqref{1.15}, tube-radius $(\la/\log\la)^{-1/2}$ is natural
for flat compact space forms when considering log-quasimodes.  Is there a natural scale for compact space forms of negative curvature?  And, finally,
do the analog of Brooks' results hold for the quasimodes defined by the spectral condition \eqref{1.1}?

\bibliography{refs}

\providecommand{\MR}[1]{}
\begin{thebibliography}{10}

\bibitem{Berard}
P.~H. B\'erard.
\newblock On the wave equation on a compact {R}iemannian manifold without
  conjugate points.
\newblock {\em Math. Z.}, 155(3):249--276, 1977.

\bibitem{bieberbach}
L.~Bieberbach.
\newblock \"{U}ber die {B}ewegungsgruppen der {E}uklidischen {R}\"{a}ume
  ({Z}weite {A}bhandlung.) {D}ie {G}ruppen mit einem endlichen
  {F}undamentalbereich.
\newblock {\em Math. Ann.}, 72(3):400--412, 1912.

\bibitem{BSTop}
M.~D. Blair and C.~D. Sogge.
\newblock Concerning {T}oponogov's theorem and logarithmic improvement of
  estimates of eigenfunctions.
\newblock {\em J. Differential Geom.}, 109(2):189--221, 2018.

\bibitem{BSSY}
J.~Bourgain, P.~Shao, C.~D. Sogge, and X.~Yao.
\newblock On {$L^p$}-resolvent estimates and the density of eigenvalues for
  compact {R}iemannian manifolds.
\newblock {\em Comm. Math. Phys.}, 333(3):1483--1527, 2015.

\bibitem{BrooksQM}
S.~Brooks.
\newblock Logarithmic-scale quasimodes that do not equidistribute.
\newblock {\em Int. Math. Res. Not. IMRN}, 22:11934--11960, 2015.

\bibitem{flat}
L.~S. Charlap.
\newblock {\em Bieberbach groups and flat manifolds}.
\newblock Universitext. Springer-Verlag, New York, 1986.

\bibitem{Germain}
P.~Germain and S.~L. Rydin~Myerson.
\newblock Bounds for spectral projectors on tori.
\newblock {\em Forum Math. Sigma}, 10:Paper No. e24, 20, 2022.

\bibitem{HassellTacy}
A.~Hassell and M.~Tacy.
\newblock {Improvement of eigenfunction estimates on manifolds of nonpositive
  curvature}.
\newblock {\em Forum Mathematicum}, 27(3):1435--1451, 2015.

\bibitem{HSqm}
X.~Huang and C.~D. Sogge.
\newblock Curvature and sharp growth rates of log-quasimodes on compact
  manifolds.
\newblock {\em preprint}.

\bibitem{sogge86}
C.~D. Sogge.
\newblock {Oscillatory integrals and spherical harmonics}.
\newblock {\em Duke Math. J.}, 53(1):43--65, 1986.

\bibitem{sogge88}
C.~D. Sogge.
\newblock {Concerning the {$L^p$} norm of spectral clusters for second-order
  elliptic operators on compact manifolds}.
\newblock {\em J. Funct. Anal.}, 77(1):123--138, 1988.

\bibitem{SoggeHangzhou}
C.~D. Sogge.
\newblock {\em {Hangzhou lectures on eigenfunctions of the {L}aplacian}},
  volume 188 of {\em {Annals of Mathematics Studies}}.
\newblock Princeton University Press, Princeton, NJ, 2014.

\bibitem{SoggeCon}
C.~D. Sogge.
\newblock Problems related to the concentration of eigenfunctions.
\newblock {\em Journés équations aux dérivées partielles, avalaible at
  http://www.numdam.org/}, 2015.

\bibitem{SFIO2}
C.~D. Sogge.
\newblock {\em Fourier integrals in classical analysis}, volume 210 of {\em
  Cambridge Tracts in Mathematics}.
\newblock Cambridge University Press, Cambridge, second edition, 2017.

\bibitem{SoggeZelNodal}
C.~D. Sogge and S.~Zelditch.
\newblock Lower bounds on the {H}ausdorff measure of nodal sets.
\newblock {\em Math. Res. Lett.}, 18(1):25--37, 2011.

\bibitem{SoggeZelditchL4}
C.~D. Sogge and S.~Zelditch.
\newblock {On eigenfunction restriction estimates and {$L^4$}-bounds for
  compact surfaces with nonpositive curvature}.
\newblock In {\em {Advances in analysis: the legacy of {E}lias {M}. {S}tein}},
  volume~50 of {\em {Princeton Math. Ser.}}, pages 447--461. Princeton Univ.
  Press, Princeton, NJ, 2014.

\bibitem{SZqm}
C.~D. Sogge and S.~Zelditch.
\newblock A note on {$L^p$}-norms of quasi-modes.
\newblock In {\em Some topics in harmonic analysis and applications}, volume~34
  of {\em Adv. Lect. Math. (ALM)}, pages 385--397. Int. Press, Somerville, MA,
  2016.

\bibitem{wolfconstant}
J.~A. Wolf.
\newblock {\em Spaces of constant curvature}.
\newblock AMS Chelsea Publishing, Providence, RI, sixth edition, 2011.

\end{thebibliography}
\bibliographystyle{abbrv}

\end{document}